\documentclass[twoside]{amsart}%
\usepackage{mathtools}
\usepackage{latexsym}
\usepackage{mathrsfs}
\usepackage{pstricks}
\usepackage{listings}

\usepackage{enumitem}
\usepackage{todonotes}
\usepackage{amssymb,bm}
\usepackage{amsmath}
\usepackage{amsfonts}
\usepackage{amsthm}
\usepackage{tikz}
\usepackage{epsfig}
\usepackage{verbatim}
\usepackage{float}
\usepackage{tabularx}

\providecommand{\U}[1]{\protect\rule{.1in}{.1in}}

\newcolumntype{Y}{>{\raggedleft\arraybackslash}X}
\def\bc{{\mathbb{C}}}

\def\bn{{\mathbb{N}}}

\def\br{{\mathbb{R}}}

\def\bz{{\mathbb{Z}}}

\def\br{\mathbb R}

\def\vs{\vskip.3cm}
\def\noi{\noindent}

\def\gdeg{G\text{\rm -deg}}

\usepackage{listings}
  \lstdefinelanguage{GAP}{
    basicstyle=\ttfamily,
    keywords={true, false, function, return, fail, if, in, while, do, od, else, elif, fi, break, continue},
    keywordstyle=\color{blue}\bfseries,
    otherkeywords={
      >, <, ==
    },
    identifierstyle=\color{black},
    sensitive=True,
    comment=[l]{\#},
    commentstyle=\color{cyan},
    stringstyle=\color{red},
    morestring=[b]',
    morestring=[b]"
  }

 \newcommand{\amal}[5]{#1\prescript{#4}{}\times_{#3}^{#5}#2}

\def\wh{\widehat}

\def\ve{\varepsilon}

\DeclareMathOperator{\id}{Id}

\newcommand\cU{\ensuremath{\mathcal U}}
\newcommand\cV{\ensuremath{\mathcal V}}
\newcommand\cW{\ensuremath{\mathcal W}}

\newtheorem{theorem}{Theorem}[section]

\newtheorem{lemma}[theorem]{Lemma}

{\bfseries}{\rmfamily}

\newtheorem{remark}[theorem]{Remark}

\newtheorem{remark-definition}[theorem]{Remark and Definition}

\begin{document}

\title[Non-Radial Solutions to Elliptic Systems]{Patterns of Non-Radial Solutions to Coupled Semilinear Elliptic Systems on a Disc}

\author[Z. Balanov --- E. Hooton --- W. Krawcewicz --- D. Rachinskii]{Zalman Balanov --- Edward Hooton --- Wieslaw Krawcewicz --- Dmitrii Rachinskii}

\address
{\textsc{Zalman Balanov}\\
Department of Mathematics\\ Xiangnan University\\
Chenzhou, Hunan 423000, China, and\\
Department of Mathematical Sciences\\
University of Texas at Dallas\\
Richardson, TX 75080, USA}

\email{balanov@utallas.edu}

\address
{\textsc{Edward Hooton}\\
Institute of Mathematics\\
Czech Academy of Sciences \\
\v{Z}itn\'a 25, 11567 Praha 1\\
Czech Republic}

\email{hooton@math.cas.cz}

\address
{\textsc{Wieslaw Krawcewicz}\\
Applied Mathematics Center at Guangzhou University\\
Guangzhou 510006, China, and\\
Department of Mathematical Sciences\\
University of Texas at Dallas\\
Richardson, TX 75080, USA.}

\email{wieslaw@utallas.edu}

\address
{\textsc{Dmitrii Rachinskii}\\
Department of Mathematical Sciences\\
University of Texas at Dallas\\
Richardson, TX 75080, USA}

\email{dmitry.rachinskiy@utdallas.edu}



\thanks{The first author Z. Balanov acknowledges the support from Xiangnan University. 
	This work has been done as a part of the Prospective Human
	Resources Support Program of the Czech Academy of Sciences, and the second author E. Hooton acknowledges the support by this program.
	The third author W. Krawcewicz acknowledges the support from National Science Foundation of China through the grant no 11871171, and the support from Guangzhou University.  The authors are also grateful to Thomas Bartsch and Jianshe Yu for their comments.}

\begin{abstract}
In this paper, we prove the existence of non-radial solutions to the problem $-\triangle u=\bm f(z,u)$, $u|_{\partial D}=0$  on the unit disc $D:=\{z\in \bc : |z|<1\}$ with   $u(z)\in \br^k$, 
where $\bm f$ is a sub-linear continuous function, differentiable with respect to $u$ at zero and satisfying   $\bm f(e^{i\theta}z,u) = \bm f(z,u)$ for all $\theta\in \br$,  $\bm f(z,-u)=- \bm f(z,u)$. Under the assumption that $\bm f$ respects additional (spacial) symmetries on
$\mathbb R^k$, we investigate symmetric properties of the corresponding non-radial solutions. The abstract result is supported by 
a numerical example with extra $S_4$-symmetries.
\end{abstract}

\maketitle


\section{Introduction} 

\noi{\bf (a) Subject.} It is a classical problem of mathematical physics to determine to which extend the solutions of elliptic BVPs inherit symmetric properties of the domain on which they are defined. For instance, the classical fixed membrane eigenvalue problem on a disc  $D$, which is 
expressed as an eigenvalue Laplace problem with  Dirichlet boundary conditions, admits a simple first eigenvalue with positive  eigenfunction exhibiting   full symmetries, while other eigenfunctions are  nodal and demonstrate symmetry breaking.  As it was established in \cite{GN}, there exists a quite large class of similar nonlinear elliptic equations for which positive solutions  exhibit full symmetries.  For a detailed exposition of how symmetric properties of  solutions (for nonlinear scalar Laplace equation) depend on geometric properties of associated energy functionals, we refer to the elegant survey  by T. Weth \cite{Weth}.

\vs
In contrast to the scalar case, systems of elliptic BVPs lead to two different  settings: variational and non-variational. In the first case, the existence of non-radial solutions was studied by many authors, see for example \cite{BDZ,BdF,BCM,CR,LWZ} and references therein.  
In the case of non-variational elliptic systems, there is a lot of interests in the existence of positive solutions. The standard method is based on  establishing  a priori bounds for positive solutions  and  using the so-called Liouville-type theorems and/or applying the  Krasnoselskii's fixed-point theorem in a cone, see \cite{CFM,C,Mon,Sou} and the references therein.

\vs

In this paper we are interested in studying the existence of non-radial solutions to the following  non-variational Laplace system on the disc  $D:=\{z\in \bc: |z|<1\}$:
\begin{equation}\label{eq:Lap}
\begin{cases}
-\triangle u=\bm f(z,u), \quad u(z)\in \br^k, \\
u|_{\partial D}=0,
\end{cases}
\end{equation}
where 
$\bm f:\overline D\times \br^k\to \br^k$ 
is a continuous odd radially symmetric  function of sublinear growth, which is differentiable at zero. To be more precise, we assume that $\bm f$  satisfies the following conditions:
  \begin{enumerate}[label=($A_\arabic*$)]
  	\item\label{c1}  $\bm f(e^{i\theta}z,u)=\bm f(z,u)$ for all $z\in D$, $u\in \br^k$ and $\theta\in \br$;
	\item\label{c2}  $\bm f(z,-u)=-\bm f(z,u)$ for all $z\in D$, $u\in \br^k$;      
	
	\item\label{c3}  there exists a $k\times k$-matrix  $A$, $c>0$  and $\beta>1$ such that 
	\begin{equation}\label{eq:s1-1}
 |\bm f(z,u)-Au|\le c|u|^\beta \quad \text{for al} \quad
 {z\in \overline{D}}, {u\in \br^k};
	\end{equation}
	\item\label{c4} there exist $a$, $b>0$ and $\alpha\in (0,1)$ such that 
		\begin{equation}\label{eq:s1-2}
		 |\bm f(z,u)|<a|u|^\alpha +b \quad \text{for al} \quad
		 {z\in \overline{D}}, {u\in \br^k}.
		\end{equation}

	\end{enumerate}
Observe that conditions \ref{c1}	 and  \ref{c2} imply that system \eqref{eq:Lap} is $G:=O(2) \times \mathbb Z_2$-symmetric.	 
\vs
Observe that given an  orthogonal $G$-representation  $V$ (here  $G$ stands for a compact Lie group) and an admissible $G$-pair  $(f,\Omega)$  in $V$
(i.e., $\Omega\subset V$ is an open bounded $G$-invariant set and $f:V\to V$ a $G$-equivariant map without zeros on $\partial \Omega$), the Brouwer degree $d_H:=\deg(f^H,\Omega^H)$ is 
well-defined for any $H \le  G$ (here $\Omega^H:= \{x \in \Omega\, :\, hx = x\; \forall h \in H\}$ 
and $f^H:= f|_{\Omega^H}$). Then, if for some $H$, one has $d_H\not=0$, the existence of solutions with symmetry at least $H$ to equation $f(x)=0$ in $\Omega$, can be predicted. Although this approach provides a  way to determine the existence of solutions in $\Omega$, and even to distinguish their different orbit types, nevertheless, it comes at a price of elaborate $H$-fixed-point space computations which can be a rather challenging task.  

\medskip

\noi{\bf (b) Method.} Our method is based on the usage of the Brouwer equivariant degree theory 
(cf. \cite{AED, survey}; see also \cite{BDZ}).  
To be more explicit, the equivariant degree $\gdeg(f,\Omega)$ is an element of the free $\bz$-module $A_0(G)$ generated by the conjugacy classes $(H)$ of subgroups $H$ of $G$ with a finite Weyl group $W(H)$:
\begin{equation}\label{eq:gdeg}
\gdeg(f,\Omega)=\sum_{(H)} n_H\, (H), \quad n_H\in \bz,
\end{equation} 
where the coefficients $n_H$ are given by the following recurrence formula
\begin{equation}\label{eq:rec}
n_H=\frac{d_H-\sum_{(L)>(H)} n_L \,n(H,L)\, |W(L)|}{|W(H)|},
\end{equation}
and  $n(H,L)$ denotes the number of subgroups $L'$ in $(L)$ such that $H\le L'$ (see \cite{AED}).  Also, we use the  notation 
\begin{equation}\label{eq:coeffi}
\text{coeff}^{H}(a):= n_H \;\ \text{for any} \;\;  a = \sum_{(H)} n_H(H) \in A_0(G). 
\end{equation}
One can immediately recognize a connection 
between the two collections: $\{d_H\}$ and  $\{ n_H\}$, where $H \le  G$ and $W(H)$ is finite. 
As a matter of fact,  $\gdeg(f,\Omega)$ satisfies the standard properties expected from any topological degree.  
However, there is one additional functorial property, which plays a crucial role in computations, namely the {\it multiplicativity property}. In fact, $A_0(G)$ has a natural structure of a ring  (which is called the {\it Burnside ring} of $G$), where the multiplication  $\cdot:A_0(G)\times A_0(G)\to A_0(G)$   is defined on generators by $(H)\cdot (K)=\sum_{(L)} m_L\, (L)$ with 
\begin{equation}\label{eq:mult}
m_L:=|(G/H\times G/K)_{(L)}/G|, \quad \text{ where } W(L) \text{ is finite}.
\end{equation}
The multiplicativity property for two admissible $G$-pairs $(f_1,\Omega_1)$ and   $(f_2,\Omega_2)$ means the following equality
\begin{equation}\label{eq:mult-property}
 \gdeg(f_1\times f_2,\Omega_1\times \Omega_2)=  \gdeg(f_1,\Omega_1)\cdot  \gdeg(f_2,\Omega_2).
\end{equation}
Given a $G$-equivariant linear isomorphism $A : V \to V$, formula \eqref{eq:mult-property} combined with the equivariant spectral decomposition of $A$, reduces the computations of $\gdeg(A,B(V))$ to the computation of the so-called basic degrees $\deg_{\cV_i}$, 
which can  be `prefabricated'
 in advance for any group $G$ (here $\deg_{\cV_i}:=\gdeg(-\id, B(\cV_i))$ with $\cV_i$ being an irreducible $ G$-representation and $B(X)$ stands for the unit ball in $X$).  

With an eye towards making our presentation acceptable for non-experts in the equivariant degree theory, we first obtain the existence
result assuming  that system \eqref{eq:Lap} respects $O(2) \times \mathbb Z_2$-symmetries only.  To show the full power of the proposed method, we next consider system  \eqref{eq:Lap} in the presence of extra spacial symmetries assuming that $\mathbb R^k$ is equipped with the structure of $\Gamma$-representation respected by $\bm f$, in which case, we  establish minimal symmetries of the corresponding non-radial solutions. The obtained abstract result is supported by a numerical example with extra $\Gamma = S_4$-symmetries. In this case,
the  computations of the related Brouwer $O(2)\times \Gamma\times \bz_2$-equivariant degrees can be assisted by GAP computer programs (the GAP code is listed in Appendix \ref{sec:GAP}). Observe also that the method can be easily adopted to treat other symmetric domains.

\medskip

\section{Functional Spaces Reformulation and {\em a priori} Bounds}
Consider the Sobolev space $\mathscr H:=H_{0}^2(D,\br^k)$ equipped with the usual norm $\|u\|:=\max\{\|D^s u\|_{L^2}: |s|\le 2\}$, where 
$s=(s_1,s_2)$, $|s|=s_1+s_2$, and $D^s u=\frac{\partial ^{|s|}}{\partial^{s_1}x\partial^{s_2}y}$.
Define the linear operator $\mathscr L:\mathscr H\to L^2(D;\br^k)$ by 
\[
\mathscr Lu:=-\triangle u, \quad u\in \mathscr H.
\]
It is well known that $\mathscr L$ is an isomorphism, i.e., the inverse operator $\mathscr L^{-1}$ is well defined and bounded. 
\vs
Choose 
\begin{equation}\label{eq:const-q}
q > \max\{1,2\alpha\} 
\end{equation}
(for example, it is enough to take $q:= 2\beta$, cf. assumptions (A3) and (A4)). Denote by $j:\mathscr H\to L^q(D;\br^k)$
the standard Sobolev embedding.  Notice that under the assumption \ref{c4},  the function
\begin{equation}\label{eq:operator-N}
N(v)(z) := \bm f(z,v(z)), \quad z\in \overline D
\end{equation}
belongs to $L^2(D,\br^k)$  for any  $v\in L^q(D;\br^k)$. Indeed, consider   $v\in L^q(D;\br^k)$.  Then, combining \eqref{eq:const-q} and  
\ref{c4} with the H\"older inequality, one has:
\begin{align}
\|N(v)\|_{L^2}&\le a \|\, |v|^\alpha\|_{L^2} +  b \sqrt \pi  \nonumber\\
& =   a \left(\int_D|v|^{2\alpha}    \right)^{\frac 12} +  b \sqrt \pi \nonumber  \\
& \le a \pi^{{1 - 2\alpha/q} \over 2} 
\left(\int_D |v|^{\frac{2\alpha q}{2\alpha }}    \right)^{\frac {2\alpha} 2q} + b \sqrt \pi \nonumber\\
&= a \pi^{1/2-\alpha/q}\|v\|^\alpha_{L^q} + b \sqrt{\pi}.
\label{eq:acting}
\end{align}
\vs
Notice that system \eqref{eq:Lap} is equivalent to the equation
\begin{equation}\label{eq:s2-1}
\mathscr L u=N(j u),\quad u\in \mathscr H,
\end{equation}
which can also be written as 
\begin{equation}\label{eq:Lap1}
\mathscr F(u):=u-\mathscr L^{-1}N(j u)=0, \quad u\in \mathscr H,
\end{equation}
with a well defined nonlinear operator $\mathscr F:\mathscr H\to \mathscr H$.  
\vs

We will need the following lemmas.

\begin{lemma}\label{lem:a-priori} Let  $\bm f:\overline D\times \br^k\to \br^k$  be a continuous function satisfying the assumption \ref{c4}, then there exists a constant $R>0$ such that $\|u\|_{\mathscr H} < R$ for any solution $u \in \mathscr H$ to system \eqref{eq:Lap}.
\end{lemma}

\begin{proof}  Assume that $u\in \mathscr H$ is a solution to \eqref{eq:Lap}, which implies that $u=\mathscr L^{-1}N(ju)$. Combining this
with $\|N(ju)\|_{L^2}\le a \pi ^{(1-\alpha)/2} \|u\|^\alpha_{L^2}+b\sqrt{\pi}$ (cf.  \ref{c4}) implies        
\begin{equation}\label{eq:est1}
\|u\|_{\mathscr H} \le a \pi ^{(1-\alpha)/2} \|\mathscr L^{-1}\|_{L^2\to \mathscr H} \, \|u\|_{L_2}^\alpha + b \sqrt \pi \|\mathscr L^{-1}\|_{L^2\to \mathscr H}.
\end{equation}
Combining \eqref{eq:est1} with $\|u\|_{\mathscr H} \ge \|u\|_{L^2}$ yields
\begin{equation}\label{eq:s2-3}
\|u\|_{L^2}\le c\|u\|^{\alpha}_{L^2}+d,
\end{equation}
where $c:= a \pi ^{(1-\alpha)/2}\|\mathscr L^{-1}\|_{L^2\to \mathscr H}$, $d:=b\sqrt{\pi}\|\mathscr L^{-1}\|_{L^2\to \mathscr H}$.
Then, since $0<\alpha<1$,  there exists $R_o>0$ such that $\psi(t):=t-ct^\alpha-d>0$ for $t\ge R_o$. Consequently, $\|u\|_{L^2}<R_o$, and by \eqref{eq:est1},
\[
\|u\|_{\mathscr H} \le  c\|u\|^{\alpha}_{L^2}+d< cR_o^\alpha+d=:R
\]
is the required constant.
\end{proof}  

\vs

We also define the linear operator $\mathscr A:\mathscr H\to \mathscr H$, by 
\[
\mathscr A(u)(z):=u-\mathscr L^{-1} Au(z), \quad u\in \mathscr H, \; z\in \overline D.
\]
\vs
\begin{lemma}\label{lem:deriv} Under the assumptions \ref{c3} and \ref{c4}, the nonlinear operator $\mathscr F : \mathscr H \to \mathscr H$ given by 
\eqref{eq:Lap1} is a completely continuous field  differentiable at $0\in \mathscr H$ with
$D\mathscr F(0)=\mathscr A$.
\end{lemma}
\begin{proof}  Take $q := 2\beta$ (cf. conditions \ref{c3}, \ref{c4} and \eqref{eq:const-q}). Then, it follows from  \eqref{eq:acting} that 
the operator $N$ given by  \eqref{eq:operator-N} takes $L^q(D;\br^k)$ to $L^2(D;\br^k)$. Combining this with continuity of $\bm f$ and the classical Krasnoselskii's theorem (see \cite{Kra1,Kra2}),  one obtains that the operator $N$ is continuous. Also, by the Rellich-Kondrachov Theorem (see, for instance, \cite{Brezis}, Theorem 9.16 and Remark 20, p. 290), $H^1_0(D; \mathbb R)$ is compactly embedded into $L^q(D;\mathbb R)$.
This together with continuous embedding of $H^2_0(D; \mathbb R)$ into $H^1_0(D; \mathbb R)$ implies a compact embedding of $\mathscr H$ into $L^q(D;\mathbb R^2)$. Thus, $\mathscr F$ is a completely continuous vector field.

Let us show that $N : L^q(D;\mathbb R^2) \to L^2(D;\mathbb R^2)$ is differentiable at the origin and the derivative coincides with the
operator generated by $A$. In fact,
\begin{align*}
\frac{1}{\|u\|_{L^q}^2}\int_D|\bm f(z,u(z))-Au(z)|^2dz&\le  \frac{c^2}{\|u\|_{L^q}^2}\int_D |u(z)|^{2\beta}dz\\
&= \frac{c^2}{\|u\|_{L^q}^2}\|u\|_{L^q}^q= c^2\|u\|_{L^q}^{2(\beta-1)},
\end{align*}
which implies that 
\[
\lim_{\|u\|_{L^q}\to 0}
\frac{\|Nu-Au\|_{L^2}}{\|u\|_{L^q}}\le c^2\lim_{\|u\|_{L^q}\to 0} \|u\|^{\beta-1}_{L^q}=0.
\]
Combining this with boundedness of the linear operators $j$ and $ \mathscr L^{-1}$ completes the proof of the lemma.
\end{proof}

\section{Existence Results}  By passing to polar coordinates, one can easily compute the spectrum of the operator $\mathscr L$ (considered as an unbounded operator in $L^2(D;\br^k)$. Namely, denote by $s_{nm}$ the $n$-th positive zero of the Bessel function $J_m$ 
of the first kind.
Define 
\[
\sigma(\mathscr L)=\{s_{nm}: n\in \bn, \; m=0,1,2,\dots\}.
\]
Then, for each eigenvalue $s_{nm}$, the corresponding eigenspace can be described as follows  (here we use the standard polar coordinates $(r,\theta)$):

\[
\mathscr E_{nm}:=\left\{J_m(s_{nm}r)\Big(\cos(m\theta)\vec a+\sin(m\theta)\vec b\Big): \vec a,\, \vec b\in \br^k\right\}.
\]
Put
\[
\mathscr H_m:=\overline{\bigoplus_{n=1}^\infty \mathscr E_{nm}}\quad \text{ and }\quad \mathscr A_m:=\mathscr A|_{\mathscr E_m}, 
\]
where the closure is taken in $\mathscr H$.
It is clear that $\mathscr A(\mathscr E_{nm})\subset \mathscr E_{nm}$, so 
$\mathscr A_{nm}:=\mathscr A|_{\mathscr E_{nm}}:\mathscr E_{nm}\to\mathscr E_{nm}$ and consequently $\mathscr A_m : \mathscr H_m\to \mathscr H_m$. 
\vs

Let us denote by $\sigma(A)$ the real spectrum of the matrix $A$ and by $\sigma_+(A)$ its positive spectrum.
We make an additional ``non-degeneracy'' assumption:
\begin{itemize}
\item[(D)] For all $\mu\in \sigma(A)$ and $n\in \bn$ and $m=0,1,2,\dots$, we have $s_{nm}\not=\mu$.
\end{itemize}
The assumption (D) implies that $\mathscr A$ is an isomorphism. Indeed, the real spectrum $\sigma(\mathscr A)$ can be easily described:
\[
\sigma(\mathscr A)=\bigcup_{m=0}^\infty \sigma(\mathscr A_m) \quad \text{ where } \; \sigma(\mathscr A_m)=\left\{ 1-\frac{\mu}{s_{nm}}: \mu\in \sigma(A),\; n\in \bn    \right\}.
\]
For every $\mu\in \sigma_+(A)$ and $m = 0,1,2,3,\dots$, 
put 
\begin{equation}\label{eq:n-m}
\mathfrak n_m(\mu):=\left|  \{(n,m): n\in \bn, \; s_{nm}<\mu\}\right|,
\end{equation}
where $|X|$ stands for the cardinality of the set $X$. As is well known 
(see, for example, \cite{Watson}, p. 486), $s_{1m} > \sqrt{m(m+2)}$, from which it follows that the numbers $\mathfrak n_m$ are non-zero only for finitely many $m=1,2,3\dots$ 
We also put
\begin{equation}\label{eq:mathfrak n-m} 
\mathfrak m_m:= \sum_{\mu\in \sigma_{+}(A)} \mathfrak n_m(\mu)\cdot m(\mu),
\end{equation}
where $m(\mu)$ stands for the algebraic multiplicity of the eigenvalue $\mu$. 
\vs

Now we can formulate our main existence result.
\begin{theorem}\label{th:main1}
Under the assumptions \ref{c1}--\ref{c4} and (D), assume that there exists $m>0$ such that $\mathfrak m_m$ is odd. Then, system \eqref{eq:Lap} admits a non-radial solution.
\end{theorem} 
\begin{proof}  Notice that the group $G:=O(2)\times \bz_2$ acts naturally on the space $\mathscr H$:
\begin{equation}\label{eq:G-action}
(h,\pm1)(u)(z):= \pm u(h^{-1}z), \quad (h \in O(2), \; \mathbb Z_2 = \{\pm 1\}, \; z \in D),
\end{equation}
and the nonlinear operator $\mathscr F$ is $G$-equivariant. By assumption (D),  the linear operator $\mathscr A:\mathscr H\to \mathscr H$ is an isomorphism,  and by Lemma \ref{lem:deriv},  there exists an $\ve>0$ such that $\mathscr F$ is $B_\ve(0)$-admissibly $G$-equivariantly homotopic to $\mathscr A$ (here $B_\ve(0)$ stands for the disc in $\mathscr H$ of radius $\ve$ centered at the origin). Similarly, by Lemma \ref{lem:a-priori} the nonlinear operator $\mathscr F$ is  $B_R(0)$-admissibly $G$-equivariantly homotopic to $\id$. Put, $\Omega:=B_R(0)\setminus \overline{B_\ve(0)}$. Then, by the additivity property of the Brouwer $G$-equivariant degree, we have 
\begin{align}
\gdeg(\mathscr F,\Omega)&=\gdeg(\mathscr F,B_R(0))-\gdeg(\mathscr F,B_\ve(0)) \nonumber\\
&=\gdeg(\id, B_R(0))-\gdeg(\mathscr A,B_\ve(0))=(G)-\gdeg(\mathscr A,B_\ve(0))  \nonumber\\
&=(G)-\prod_{m=0}^\infty \gdeg(\mathscr A_m,B(\mathscr H_m))\nonumber\\
&=(G)-\gdeg(\mathscr A_0,B(\mathscr E_0))\cdot \prod_{m=1}^\infty (\deg_{\mathcal V^-_m})^{\mathfrak m_m},\label{eq:product-basic}
\end{align}
where $\deg_{\mathcal V^-_m}:=\gdeg(-\id,B(\mathcal V_m^{-}))$ 
and $\mathcal V^-_m$ stands for the $m$-th irreducible $G$-representation. Notice that $\gdeg(\mathscr A_0,B(\mathscr E_0))$ is equal either to $(G)$ or to $(G)-(O(2))$ and 
$\deg_{\mathcal V^-_m}=(G)-(D_{2m}^d)$ for $m\in \bn$.  It follows directly from the properties of the Burnside ring multiplication that if $\mathfrak m_m$ is odd, then $\gdeg(\mathscr F,\Omega)\not=0$. More precisely, there exists an $m'\in \bn$ such that $m$ divides $m'$ and coeff$^{D_{2m'}^d}(\gdeg(\mathscr A,B(\mathscr H)) \not=0$ (for more details pertinent to this fact, we refer to Section \ref{sec:app4.3}, where a more involved symmetric setting is discussed). Therefore, by the existence property of the Brouwer equivariant degree, there exists a non-zero solution $u$ to the equation $\mathscr F(u)=0$ such that $G_u\ge D_{2m'}^d$. Clearly, such a solution $u$ is not a radial solution.
\end{proof} 

\vs
Let us point out that by applying an appropriate $H$-fixed point reduction, it is possible to remove the condition (D) from Theorem \ref{th:main1}. To be more precise, put
\begin{equation*}
\mathscr C:=\{m\in \bz_+: \exists_{n\in \bn}\; \exists_{\mu\in \sigma(A)} \;\text{ such that }\; s_{nm}=\mu\}.
\end{equation*}
Since $\mathscr A$ is a Fredholm operator, the set $\mathscr C$ is finite. Put $H_l:=\bz_{2l}^d$. Then, 
\[
\mathscr H^{H_l} = \overline{\bigoplus_{m=1}^\infty  \mathscr H_{(2m-1)l}}.
\]
Assume that there exists  $l\in \bn$ such that 
\begin{equation}\label{eq:s3-1}
\mathscr C\cap \{ (2m-1)l: m\in \bn\}=\emptyset.
\end{equation}
Then, $\mathscr A^{H_l}:=\mathscr A|_{\mathscr E^{H_l}}:\mathscr H^{H_l}\to \mathscr H^{H_l}$ is an isomorphism and we obtain the following theorem.
\vs

\begin{theorem}\label{th:main2}
Under the assumptions \ref{c1}--\ref{c4}, let $l\in \bn$ be an integer such that  \eqref{eq:s3-1} is satisfied. If there exists $m\in \bn$ such that $\mathfrak m_{(2m-1)l}$ is odd, then system \eqref{eq:Lap} admits a non-radial solution.
\end{theorem} 
\begin{proof} One needs to repeat the same  arguments that were  presented in the proof of Theorem~\ref{th:main1}.
\end{proof}
\medskip
\vs

\vs
\section{Semilinear Elliptic Systems on a Disc with Additional Symmetries}\label{sec:app4.3} 
\noindent
{\bf (a) Symmetrically Interacting Systems.}  Let us consider a system composed of $n$-coupled identical  systems of type \eqref{eq:Lap}\, which we can write as follows:
\begin{equation}\label{eq:Lap-symm1}
\begin{cases}
-\triangle u_1=f(z,u_1) + g_1(u_1,u),\\
-\triangle u_2=f(z,u_2) + g_2(u_2,u),\\
\vdots\\
-\triangle u_n=f(z,u_n) + g_n(u_n,u),\\
u_1|_{\partial D}=u_2|_{\partial D}=\dots =u_n|_{\partial D}=0,
\end{cases}
\end{equation}
where $f$ is similar to $\bm f$ in \eqref{eq:Lap}, $u_j(z)\in \br^s$, $u=(u_1,u_2,\dots, u_n)$ and the functions $g_j(u_j,u)$ describe the interaction 
between the $j$-th function $u_j$ and other functions. As an example, consider a configuration of $8$ such systems coupled in a cube fashion (see figure below). 
Then system \eqref{eq:Lap-symm1} can be written as:
\begin{equation}\label{eq:Lap-symm2}
\hskip4cm\begin{cases}
-\triangle u_1=f(z,u_1)+g(u_2-u_1,u_4-u_1,u_6-u_1),\\
-\triangle u_2=f(z,u_2)+g(u_1-u_2,u_3-u_2,u_7-u_2)\\
-\triangle u_3=f(z,u_3)+g(u_2-u_3,u_4-u_3,u_8-u_3),\\
-\triangle u_4=f(z,u_4)+g(u_3-u_4,u_5-u_4,u_6-u_4)\\
 -\triangle u_5=f(z,u_5)+g(u_4-u_5,u_6-u_5,u_8-u_5)\\
 -\triangle u_6=f(z,u_6)+g(u_1-u_6,u_5-u_6,u_7-u_6)\\
 -\triangle u_7=f(z,u_7)+g(u_2-u_7,u_6-u_7,u_8-u_7),\\
 -\triangle u_8=f(z,u_8)+g(u_3-u_8,u_5-u_8,u_7-u_8),\\\
u_1|_{\partial D}=u_2|_{\partial D}=\dots =u_8|_{\partial D}=0,
\end{cases}
\end{equation}
where the continuous function $g:\br^s\times \br^s\times \br^s\to \br^s$ is such that $g(x_1,x_2,x_3)=g(x_2,x_3,x_1) = g(x_1,x_3,x_2)$ for all $x_1,x_2,x_3\in \br^k$ and $g(0,0,0)=0$. We will also assume that $g$ is differentiable at $(0,0,0)$, thus there exists an $s\times s$-matrix $B$ such that 
\[
Dg(0,0,0)(x_1,x_2,x_3)=B(x_1+x_2+x_3),\quad  x_1,\, x_2,\,  x_3\in \br^s.
\]  
\vs
Keeping the exemplary system \eqref{eq:Lap-symm2} in mind, assume that $\Gamma$ is a subgroup of the permutation group $S_k$ acting on $V:=\br^k$ by permuting the coordinates of vectors in $\br^k$. Consider system \eqref{eq:Lap} assuming, in addition to conditions \ref{c1}--\ref{c4}, that the following condition is satisfied:
  \begin{enumerate}[label=($B_\arabic*$)]
\item\label{d5} the function $\bm f$ is $\Gamma$-equivariant, i.e. for every $\gamma\in \Gamma$ and $x\in V$, $\bm f(z, \gamma x)=\gamma \bm f(z,x)$ (for all $z\in D$).
	\end{enumerate}
We are interested in extending Theorems \ref{th:main1} and \ref{th:main2} to the symmetric setting providing an additional information on symmetric properties of non-radial solutions.

\medskip
\noi {\bf (b) Equivariant Setting in Functional Spaces.} Put $G:= O(2)\times \Gamma\times \bz_2$ and consider the Hilbert $G$-representation 
$\mathscr H:=H_o^2(D;V)$, with the $O(2) \times \mathbb Z_2$-action given by \eqref{eq:G-action} and the
$\Gamma$-action given by $(\gamma u)(z):=\gamma u(z)$, $\gamma\in \Gamma$, $z\in D$, $u\in \mathscr H$. 
Then, the operator $\mathscr F$ given by \eqref{eq:Lap1} is $G$-equivariant.
\vs
Consider the $\Gamma$-isotypic decomposition of $V$
\[
V=V_0\oplus V_1\oplus V_2\oplus \dots \oplus V_{\mathfrak r}, 
\]
where the $\Gamma$-isotypic component $V_j$ is modeled on the irreducible $\Gamma$-representation $\cU_j$. We also put
\begin{equation}\label{eq-m-j-mult}
m_j:=\text{dim\,} V_j/\text{dim\,} \cU_j, \quad j\in \{0,1,\dots,\mathfrak r\}.
\end{equation}
To simplify the computations, we introduce an additional condition (cf. condition \ref{c3}): 
\vs
\begin{itemize}
\item[($B_2$)] For every  $j\in \{0,1,2,\dots , \mathfrak r\}$, $A|_{V_j} =\mu_j\, \id_{V_j}$, for some $\mu_j\in \br$.
\end{itemize}

\vs 
The $G$-isotypic decomposition of $\mathscr H$ can be easily constructed. For every $j\in \{0,1,2,\dots , \mathfrak r\}$ and $(n,m)\in \bn\times \bn\cup \{0\}$  we put 
\[
\mathscr E^j_{nm}:=\left\{J_m(s_{nm}r)\Big(\cos(m\theta)\vec a+\sin(m\theta)\vec b\Big): \vec a,\, \vec b\in V_j\right\},
\]
and 
\[
\mathscr H_{m,j}:=\overline{\bigoplus_{n=1}^\infty \mathscr E^j_{nm}}.
\]
Then, we obtain the following $G$-isotypic decomposition of $\mathscr H$:
\begin{equation}\label{eq:isoH}
\mathscr H= \bigoplus _{j=0}^{\mathfrak r} \overline{\bigoplus_{m=0}^\infty \mathscr H_{m,j}}
\end{equation}
(here the $G$-isotypic component $\mathscr H_{m,j}$ is modeled on the irreducible $G$-representation $\cV^-_{m,j}=\cW_m\otimes \cU_j^-$, where $\cW_m$ stands for the $m$-th irreducible $O(2)$-representation and $\cU_j^-$ is the representation $\cU_j$ with the antipodal $\bz_2$-action). 

\vs

In order to formulate and prove one of our main results in the symmetric setting, let us discuss some additional properties of the basic degrees. To begin with,  take $\nu \in \mathbb N$ and define the Lie group homomorphism 
$\psi_\nu: O(2)\times \Gamma\times \bz_2\to O(2)\times \Gamma\times \bz_2$ by
\begin{equation}\label{eq:app4.3-fold}
\psi_\nu(e^{i\theta},\gamma,\pm 1)=(e^{i\nu\theta},\gamma, \pm 1), \quad \psi_\nu(e^{i\theta}\kappa,\gamma, \pm 1)=(e^{i\nu\theta}\kappa,\gamma,\pm 1)
\end{equation} 
(sometimes $\psi_\nu$ is called a {\it $\nu$-folding} homomorphism). 
This homomorphism induces the Burnside ring homomorphism $\Psi_\nu:A_0(G)\to A_0(G)$, where for a generator  $(H)\in \Phi_0(G)$, one has 
\begin{equation}\label{eq:nu-fold-hom}
\Psi_\nu(H)=(H_\nu), \quad H_\nu = \psi_\nu^{-1}(H).
\end{equation}
In particular, for any $j=\{0,1,\dots, \mathfrak r\}$ and  $m\ge 0$,  one has 
\begin{equation}\label{eq:folding-basic-degree}
\Psi_\nu\big(\deg_{\cV^-_{m,j}}  \big) = \deg_{\cV^-_{\nu m,j}}.
\end{equation}
\vs
\begin{remark}\label{rem-non-maximal}
{\rm If $(H)$ is a maximal orbit type in $\mathscr  H \setminus \{0\}$, then necessarily one has that $(H)$ is an orbit type in
$\mathscr  H_{0,j} \setminus \{0\}$ for some $j \in \{1,...,\mathfrak r\}$.  Hence, in this case, $(H)$ is an orbit type of a radially symmetric map (by the same token, $\Phi_\nu(H) = (H)$). Hence, there do {\it not} exist maximal orbit types in  $\mathscr  H \setminus \{0\}$ corresponding to a non-radial map. This suggests to consider maximal orbit types in $\mathscr  H_{m,j} \setminus \{0\}$ with $m > 0$.}
\end{remark}

For $m > 0$, denote by $\mathfrak M_m$  the set of all maximal orbit types $(H)$ in $\Phi_0(G; \mathscr H_m\setminus\{0\})$. 
Clearly, $\Psi_\nu(\mathfrak M_m) = \mathfrak M_{m\nu}$ for any $\nu \in \mathbb N$ (in particular, 
$\Psi_\nu(\mathfrak M_1) = \mathfrak M_\nu)$. Being motivated by Remark \ref{rem-non-maximal} and formula \eqref{eq:product-basic}, and in order to simplify our exposition, consider the isotypic decomposition 
\begin{equation}\label{eq:isotyp-H-1}
\mathscr H_1:=\mathscr H_{1,0}\oplus \mathscr H_{1,1}\oplus \dots\oplus \mathscr H_{1,\mathfrak r}. 
\end{equation}
Decomposition \eqref{eq:isotyp-H-1} together with formula \eqref{eq:nu-fold-hom} (see also \eqref{eq:coeffi}, \eqref{eq-m-j-mult} and condition (B$_2$)) allow us to refine formulas \eqref{eq:n-m} and \eqref{eq:mathfrak n-m} as follows: for a given $(H)\in \mathfrak M_1$ and $\nu\in \bn$, 
define
\begin{equation}\label{eq:n-m-equivar}
\mathfrak n_j(H_\nu) := \left| \{n\in \bn: s_{n\nu}<\mu_j,\; \text{coeff}^{H_\nu}(\deg_{\cV^-_{\nu,j}})\not=0\}     \right|, 
\end{equation}
and
\begin{equation}\label{eq:mathfrak n-m-equivar}
\mathfrak m(H_\nu) :=\sum_{j=0}^{\mathfrak r} \mathfrak n_j(H_\nu) m_j.
\end{equation}
We are now in a position to formulate our first main result in the symmetric setting.
\begin{theorem}\label{th:main1-3}
Under the assumptions \ref{c1}--\ref{c4}, (D), (B$_1$)--(B$_2$), suppose that there exist $(H)\in \mathfrak M_1$ and  $\nu > 0$ such that $\mathfrak m(H_\nu)$ is odd (cf. \eqref{eq:n-m-equivar}--\eqref{eq:mathfrak n-m-equivar}. Put 
\begin{equation}\label{eq:nu-o}
\nu_o = \nu_o(H) := \max  \{\nu^{\prime} \; : \; \mathfrak m(H_{\nu^{\prime})}\;  \text{is odd}\} 
\end{equation} 
(obviously, $\nu_o < \infty$).
Then, system \eqref{eq:Lap} admits 
a non-radial solution $u\in \mathscr H$ such that $G_u=H_{\nu_o r}$ for some $r\in \bn$.

\vs

\end{theorem} 
\begin{proof} Let $B_\ve(0)$, $B_R(0)$ and $\Omega$ be the same as in the proof of Theorem \ref{th:main1}. Then, $\mathscr F$
is $B_\ve(0)$-admissibly $G$-equivariantly homotopic to the isomorphism $\mathscr A$ and $B_R(0)$-admissibly $G$-equivariantly homotopic to $\id$. Moreover, formula \eqref{eq:product-basic} refines to the following one:
\begin{align}
\gdeg(\mathscr F,\Omega) &= \gdeg(\mathscr F,B_R(0))-\gdeg(\mathscr F,B_\ve(0)) \nonumber\\
&= (G)-\prod_{j=0}^{\mathfrak r} \prod_{m=0}^\infty \prod_{n=1}^\infty \gdeg(\mathscr A^j_{n,m},B(\mathscr E^j_{nm})), 
\label{eq:product-basic-refined}
\end{align}
where $\mathscr A^j_{n,m}:=\mathscr A|_{\mathscr E^j_{nm}}$.  Notice that 
\begin{equation}\label{eq:deg-through-basic}
\gdeg(\mathscr A^j_{n,m},B(\mathscr E^j_{nm}))=\begin{cases}
(\deg_{\cV^-_{m,j}})^{m_j}& \text{ if } s_{nm}<\mu_j,\\
(G) &\text{ otherwise}. 
\end{cases}
\end{equation}
Also, for any basic degree $\deg_{\cV^-_{m,j}}$ and {\it maximal} orbit type $(H_o)$ in $\cV^-_{m,j} \setminus \{0\}$, the recurrence 
formula \eqref{eq:rec} implies
\begin{equation}\label{eq:max-type-basic-deg}
\deg_{\cV_{m,j}^-}=(G)- x_o(H_o) + c,\quad     -x_o:=\frac{(-1)^{\text{dim}\cV_{m,j}^{-H_o}}-1}{|W(H_o)|},
\end{equation}
where $c \in A(G)$ satisfies the condition: coeff$^{H_o}(c) = 0$. 
Then, by  \eqref{eq:max-type-basic-deg}, one has  
\begin{equation}\label{eq:coef-x-o-ireduc}
x_o=\begin{cases}
0 & \; \text{if $\text{dim}\cV_{m,j}^{-H_o}$ is even};\\ 
1 & \text{ if $\text{dim}\cV_{m,j}^{-H_o}$ is odd  and $|W(H_o)|=2$};\\
2 &    \text{ if $\text{dim}\cV_{m,j}^{-H_o}$ is odd  and $|W(H_o)|=1$}.
\end{cases}
\end{equation} 
To complete the proof of Theorem  \ref{th:main1-3}, we need the following important
\begin{lemma}\label{lem:same-maximal}
Suppose that $(H_o)$ is a maximal orbit type in $\cV_{m,j}^-\setminus\{0\}$ and $\cV_{m',j'}^-\setminus \{0\}$ and both 
$\cV_{m,j}^{-H_o}$ and  $\cV_{m',j'}^{-H_o}$ are of odd dimension. Then:
\begin{itemize}
\item[(i)] \rm{coeff}$^{H_o}(\deg{\cV_{m,j}^-})$ =  \rm{coeff}$^{H_o}(\deg{\cV_{m',j'}^-})$; 
\item[(ii)] \rm{coeff}$^{H_o}(\deg{\cV_{m,j}^-}\cdot \deg{\cV_{m',j'}^-}) = 0$. 
\end{itemize}
\end{lemma}

\noindent
{\it Proof of Lemma \ref{lem:same-maximal}:}
\

(i) Follows immediately from \eqref{eq:coef-x-o-ireduc}.
\smallskip

(ii) Consider the product
\begin{align*}
\deg{\cV_{m,j}^-}\cdot \deg{\cV_{m',j'}^-}&=  \big( (G)-x_o(H_o)+ c   \big) \cdot \big( (G)-x_o(H_o)+ c'    \big)\\
&= (G)-2x_o(H_o)+x_o^2(H_o)\cdot (H_o) + b,
\end{align*}
where $c$, $c'$, $b\in A(G)$ satisfy coeff$^{H_o}(c)$ = coeff$^{H_o}(c')$ = coeff$^{H_o}(b)$ = 0.  Then, by using the recurrence formula \eqref{eq:rec}, one obtains:
  \[
  (H_o)\cdot(H_o)=y_o(H_o)+ d, \quad y_o:=\frac{n(H_o,H_o)^2|W(H_o)|^2}{|W(H_o)|}= |(W(H_o)|,
  \]
with coeff$^{H_o}(d) = 0$.
 Hence,
  \[
  -2x_o(H_o)+x_o^2(H_o)\cdot (H_o)=\begin{cases}
 -2+2 & \text{if } x_o=1, \;\; |W(H_o)|=2\\
  -4+4  & \text{if } x_o=2, \;\; |W(H_o)|=1\\
  \end{cases}\cdot (H_o)
  =0,
  \]
and the statement follows.

\bigskip
\noindent
{\it Completion of the proof of Theorem \ref{th:main1-3}}. 
By \eqref{eq:product-basic-refined}-\eqref{eq:max-type-basic-deg} (see also \eqref{eq:nu-o}), one has: 
\begin{equation}\label{eq:100}
\prod_{j=0}^{\mathfrak r} \prod_{m=0}^\infty \prod_{n=1}^\infty \gdeg(\mathscr A^j_{n,m},B(\mathscr E^j_{nm})) = 
\big((G) - a\big) \cdot  \prod_{k=1}^{\mathfrak m(H_{\nu_o})} \big((G)+x_o(H_{\nu_o}) + c_k\big),
\end{equation}
where coeff$^{H_{\nu_o}}(c_k)$ = coeff$^{H_{\nu_o}}(a) = 0$. Since $ \mathfrak m(H_{\nu_o})$ is odd,  Lemma \ref{lem:same-maximal} 
yields:
\begin{align}
\prod_{k=1}^{\mathfrak m(H_{\nu_o})} \big((G)+x_o(H_{\nu_o}) + c_k\big) &= 
\prod_{k=2}^{\mathfrak m(H_{\nu_o})} \big((G)+x_o(H_{\nu_o}) + c_k\big) \cdot  \big((G)+x_o(H_{\nu_o}) + c_1\big) \nonumber\\ 
&= \big((G) - a^{\prime}\big) \cdot \big((G)+x_o(H_{\nu_o}) + c_1\big), \label{eq:101}
\end{align}
where coeff$^{H_{\nu_o}}(a^{\prime}) = 0$. Combining \eqref{eq:product-basic-refined}, \eqref{eq:100} and \eqref{eq:101} with the maximality of $(H_{\nu_o})$ yields:
\begin{align*}
\gdeg(\mathscr F,\Omega) &= (G) - \big((G) - b \big)  \cdot \big((G)+x_o(H_{\nu_o}) + c_1\big) \\
&=  x_o(H_{\nu_o}) - b - x_o (H_{\nu_o}) \cdot b - b - b \cdot c_1 + c_1 = x_o(H_{\nu_o}) + d, 
\end{align*}
(here coeff$^{H_{\nu_o}}(b)$ = coeff$^{H_{\nu_o}}(d) = 0$).
Therefore, by the existence property of the Brouwer equivariant degree, there exists a non-zero solution $u$ to the equation $\mathscr F(u)=0$ such that $G_u\ge H_{\nu_o}$. Clearly, such a solution $u$ is not a radial solution.
\end{proof} 
\vs
Similarly to the non-equivariant case, one can formulate the second main result providing the existence of non-radial solutions without assuming the condition (D). 
\vs

\begin{theorem}\label{th:main2-3}
Under the assumptions \ref{c1}--\ref{c4}, (B$_1$)--(B$_2$), let $l\in \bn$ be an integer such that  \eqref{eq:s3-1} is satisfied. 
Suppose, in addition, that there exist $(H)\in \mathfrak M_1$ and  $m > 0$ such that $\mathfrak m(H_{(2m-1)l})$ is odd (cf. \eqref{eq:n-m-equivar}--\eqref{eq:mathfrak n-m-equivar}. Put 
\begin{equation}\label{eq:nu-o}
\nu_o:= \max  \{(2m-1)l \; : \; \mathfrak m(H_{(2m-1)l})\;  \text{is odd}\} 
\end{equation} 
(obviously, $\nu_o < \infty$).
Then, system \eqref{eq:Lap} admits 
a non-radial solution $u\in \mathscr H$ such that $G_u=H_{\nu_o r}$ for some $r\in \bn$.

\end{theorem} 

\medskip\noi
{\bf (c) Example.}  Assume that $\Gamma=S_4$ acts on the space $V:=\br^8$ by permuting the coordinates of vectors the same way as the symmetries of a cube permute its vertices. We consider system   \eqref{eq:Lap} for which the map $\bm f$ satisfies  conditions  \ref{c1}--\ref{c4}, (B$_1$), and the matrix $A$ is given by 
\begin{equation}
A=\left[\begin{array}{cccccccc}~~ c~~& ~~d~~& ~~0~~& ~~d~~&~~0~~&~~d~~&~~0~~&~~0~~  \\
d& c& d& 0&0&0&d&0  \\
0& d& c& d&0&0&0&d  \\
d& 0& d& c&d&0&0&0  \\
0& 0& 0& d&c&d&0&d  \\
d& 0& 0& 0&d&c&d&0  \\
0& d& 0& 0&0&d&c&d  \\
0& 0& d& 0&d&0&d&c  
\end{array} \right] \label{eq:oct-A}
\end{equation}
(the values of $c$ and $d$ will be specified later on).
Notice that one has the following table of characters:
\medskip
\begin{center}
\begin{tabular}{|c|ccccc|}
\hline
$\chi$ & $()$ & $(1,2)$ & $(1,2)(3,4)$ & $(1,2,3)$ & $(1,2,3,4)$\\
\hline
$\chi_0$ &1&1&1&1&1\\
$\chi_1$ &1&-1&1&1&-1\\
$\chi_2$ &2&0&2&-1&0\\
$\chi_3$ &3&1&-1&0&-1\\
$\chi_4$ &3&-1&-1&0&1\\
\hline
$\chi_V$ &8&0&0&2&0\\
\hline 
 \end{tabular} 
\end{center}
\medskip
which implies that 
\begin{equation}\label{isotyp-S4}
V=\cU_0\oplus\cU_1\oplus \cU_3\oplus \cU_4
\end{equation}
(here $\cU_j$ stands for the irreducible representation corresponding to the character $\chi_j$). Clearly, \eqref{isotyp-S4} implies
 condition (B$_2$). One can easily compute the spectrum of $A$:
\[\sigma(A)=\{\mu_0=c+3d,\mu_1=c-3d,\mu_3=c+d,\mu_4=c-d\}.\]
Obviously,
\[
\text{dim\,} E(\mu_0)=\text{dim\,} E(\mu_1)=1,\quad \text{ and } \text{dim\,} E(\mu_3)=\text{dim\,} E(\mu_4)=3
\]
and 
\[
E(\mu_0)=\cU_0, \;\; E(\mu_1)=\cU_1,\;\; E(\mu_3)=\cU_3,\;\; E(\mu_4)=\cU_4.
\]
The numbers $s_{nm}$ are shown in the table below:

\medskip

\begin{center}
\begin{tabular}{|c|cccccc|}
\hline
$n$&	 $J _0(x)$	&$J_1(x)$&$	J_2(x)$&$	J_3(x)$&$	J_4(x)$&$	J_5(x)$\\
\hline
1&	2.4048&	3.8317&	5.1356&	6.3802&	7.5883&	8.7715\\
2&	5.5201&	7.0156&	8.4172&	9.7610&	11.0647&	12.3386\\
3&	8.6537&	10.1735&	11.6198&	13.0152&	14.3725&	15.7002\\
4&	11.7915&	13.3237&	14.7960&	16.2235&	17.6160&	18.9801\\
5&	14.9309&	16.4706&	17.9598&	19.4094&	20.8269&	22.2178\\

\hline
\end{tabular}
\end{center}
Take  $c=4$ and $d=1$, thus $\mu_0=7$, $\mu_1=1$, $\mu_3=5$, $\mu_4=3$ and condition (D) is satisfied.
 
 Put $G:=O(2)\times S_4\times \bz_2$.  Then,
\begin{align*}
\gdeg(\mathscr F,\Omega)&=(G)-\deg_{\cV_{0,0}^-}^2\cdot \deg_{\cV_{1,0}^-}\cdot \deg_{\cV_{2,0}^-}\cdot \deg_{\cV_{3,0}^-}\cdot \deg_{\cV_{0,3}^-}\cdot \deg_{\cV_{1,3}^-}\cdot \deg_{\cV_{0,4}^-}\\
&=(G)-\deg_{\cV_{1,0}^-}\cdot \deg_{\cV_{2,0}^-}\cdot \deg_{\cV_{3,0}^-}\cdot \deg_{\cV_{0,3}^-}\cdot \deg_{\cV_{1,3}^-}\cdot \deg_{\cV_{0,4}^-}.
\end{align*}
In order to effectively apply Theorem \ref{th:main1-3}, we used GAP package {\tt EquiDeg} to carry on all the related symbolic computations.  We refer to Appendix \ref{sec:GAP} for the exact GAP codes that were applied for this example. In particular, 
the maximal orbit types in $\mathfrak M_1$ are:
\begin{alignat*}{5}
(H_{1,228})&=(D_6\times_{D_6}D_3^p),\quad &(H_{1,248})&= (D_4\times_{D_4}^{\bz_2^-} D_4^p),\quad &(H_{1,286})&=(D_2^{D_1}\times^{D_2^d}_{\bz_2} D_2^p),\\
(H_{1,334})&=(D_2^{D_1}\times_{\bz_2} ^{D_4^z}D_ 4^p),\quad &(H_{1,333})&=(D_2^{D_1}\times_{\bz_2}^{D_4^d} D_4^p),\quad& (H_{1,360})&=(D_2^{D_1}\times_{\bz_2}^{S_4} S_4^p),\\
(H_{1,359})&=(D_2^{D_1}\times _{\bz_2}^{S_4^-}S_4^p).&&&&
\end{alignat*}

\medskip
\begin{remark}\label{rem:notations} \rm (i)  For any subgroup $S  \leq S_4$, the symbol $S^p$ stands for $S \times \mathbb Z_2$.  

(ii) Given two subgroups $H \leq O(2)$ and $K \leq S_4^p$, we refer to Appendix \ref{subsec:G-degree}, item (a), for the ``amalgamated notation"  $\amal{H}{Z} {L} {R}{K}$. 

(iii) We refer to  \cite{AED} for the explicit description of the (sub)groups $S_4^-$, $D_k^z$, $D_k^d$, and $\mathbb Z_2^-$.
\end{remark}

Moreover, we have 
\begin{align*}
 \deg_{\cV_{1,0}^-}&=(G)-({\bm H_{1,360}})\\
\deg_{\cV_{1,3}^-}&=(G)-2(H_{1,46})-(H_{1,83})+2(H_{1,108})+2(H_{1,126}) +2(H_{1,129}) +2(H_{1,204})\\
&+(H_{1,212})-2(\bm{H_{1,228}})-2(\bm{H_{1,248}})-(\bm{H_{1,286}})-(H_{1,330})-(\bm{H_{1,334}}),
\end{align*}
where we denote in bold the maximal orbit types from $\mathfrak M_1$. Then, by inspection, one can easily determine that 
\begin{equation}\label{eq:numbers-m(H)-required}
\mathfrak m({\bm H_{1,360}})=1, \quad \mathfrak m (\bm{H_{1,228}})=1, \quad \mathfrak m(\bm{H_{1,248}})=1, \quad \mathfrak m(\bm{H_{1,286}})=1, \quad \mathfrak m(\bm{H_{1,334}})=1.
\end{equation}
Since, the  exact degree $\gdeg(\mathscr F,\Omega) $ can be  effectively computed using the GAP package, we can use it to double check the correctness of our conclusions: 
\begin{align*}
\gdeg(\mathscr F,\Omega) &=-2(H_{1,2})-2(H_{1,5})-2(H_{1,8})-8(H_{1,11})-2(H_{1,14})+4(H_{1,20})\\
&+4(H_{1,21})+2(H_{1,25})+2(H_{1,41})+2(H_{1,46})]+4(H_{1,51})+2(H_{1,53})\\
&+2(H_{1,57})+2(H_{1,59})-2(H_{1,77})+1(H_{1,83})-(H_{1,88})-(H_{1,89})+2(H_{1,90})\\
&+(H_{1,91}))-(H_{1,92})+2(H_{1,98})+2(H_{1,102})+2(H_{1,103})-(H_{1,106})\\
&-2(H_{1,108})-2(H_{1,115})-2(H_{1,117})-2(H_{1,119})-2(H_{1,126})-2(H_{1,129})\\
&+2(H_{1,152})+2(H_{1,155})-2(H_{1,204})-(H_{1,208})-(H_{1,212})+(H_{1,220})\\
&+(H_{1,221})-(H_{1,223})+(H_{1,224})-2(H_{1,227})-2(H_{1,228})+(H_{1,244})+(H_{1,245})\\
&-(H_{1,246})+2(H_{1,248})-2(H_{1,256})+(H_{1,286})-(H_{1,292})+(H_{1,295})\\
&+(H_{1,298})-(H_{1,300})-(H_{1,301})+(H_{1,330})+(H_{1,334})-(H_{1,360})-(H_{2,17})\\
&+(H_{2,88})+(H_{2,89})+(H_{2,106})+(H_{2,223})-(H_{2,244})-(H_{2,245})-(H_{2,295})\\
&-(H_{2,298})+(H_{2,360})-(H_{3,17})+(H_{3,88})+(H_{3,89})+(H_{3,106})+(H_{3,223})\\
&-(H_{3,244})-(H_{3,245})-(H_{3,295})-(H_{3,298})+(H_{3,360})+(H_{0,12})-(H_{0,43})\\
&-(H_{0,44})-(H_{0,52})-(H_{0,90})+(H_{0,98})+(H_{0,99})+(H_{0,110})+(H_{0,111}).
\end{align*}
On the other hand, by analyzing the coefficients of $\gdeg(\mathscr F,\Omega) $, one can also deduct the existence of various symmetric types of {\bf radial solutions}. More precisely, the maximal orbit types in $\displaystyle \bigoplus_{j=0}^{\mathfrak r} \mathscr H_ {0,j}\setminus\{0\}$ with non-zero coefficients in $\gdeg(\mathscr F,\Omega)$ are:
\begin{alignat*}{5}
(H_{0,84})&= (O(2)\times D_2^d), \quad &(H_{0,98})&=(O(2)\times D_3), \quad &(H_{0,99})&=(O(2)\times D_3^z)\\
(H_{0,110})&=(O(2)\times D_ 4^z),\quad &(H_{0,111})&=(O(2)\times D_4^d).&&
\end{alignat*}
Combing this with \eqref{eq:numbers-m(H)-required} and Theorem \ref{th:main1-3}, one obtains the following
 
 \vs

\begin{theorem} Under the assumptions \ref{c1}--\ref{c4}, (B$_1$)--(B$_2$) assume that the matrix $A$ is given by \eqref{eq:oct-A} with  $c=4$ and $d=1$. 
The  system \eqref{eq:Lap} admits at least five different orbits of 
non-radial solutions $u\in \mathscr H$ with the following orbit types $(G_u)$:
\begin{itemize}
\item 
$(D_{6m}^{\bz_m}\times_{D_6}D_3^p)$ for some $m\in \bn$;
\item
$(D_{4m}^{\bz_m}\times_{D_4}^{\bz_2^-} D_4^p)$  for some $m\in \bn$;

\item $(D_{2m}^{D_m}\times^{D_2^d}_{\bz_2} D_2^p)$  for some $m\in \bn$;

\item $(D_{2m}^{D_m}\times_{\bz_2} ^{D_4^z}D_ 4^p)$  for some $m\in \bn$;
\item $(D_{2m}^{D_m}\times_{\bz_2}^{S_4} S_4^p)$  for some $m\in \bn$.
\end{itemize}
Moreover, the system \eqref{eq:Lap} admits at least four different orbits of  
radial solutions $u\in \mathscr H$ with the  following orbit types $(G_u)$:
\[
(O(2)\times D_3),\quad 
(O(2)\times D_3^z),\quad 
(O(2)\times D_ 4^z), \quad
(O(2)\times D_4^d).
\]
\end{theorem} 
\begin{proof} The statement follows directly from Theorem \ref{th:main1-3}. To be more precise, notice that for the orbit types  $(H)$ with $H= H_{1,228}$, $H_{1,248}$,  $H_{1,286}$ and  $H_{1,334}$, we have $\nu=1$ and $\nu_o=1$, and for $H= H_{1,360}$, we have $\nu=1$ and $\nu_o=3$. On the other hand, for the maximal orbit types 
$(O(2)\times D_3)$,
$(O(2)\times D_3^z)$,
$(O(2)\times D_ 4^z)$ and $(O(2)\times D_4^d)$, one can conclude the existence of the related radial solutions  to  \eqref{eq:Lap}  by a direct inspection of the explicit formula for $\gdeg(\mathscr F,\Omega) $, where it is evident that all these orbit types appear with  non-zero coefficients.
\end{proof}
\vs
\appendix
\section{Equivariant Brouwer Degree Background}
\label{subsec:G-degree}

\noi
{\bf (a) Amalgamated Notation.} 
Given two groups $G_{1}$ and
$G_{2}$, 
the well-known result of \'E. Goursat (see \cite{Goursat}) provides the following description of a
subgroup $\mathscr U \leq G_{1}\times G_{2}$: 
there exist subgroups
$H\leq G_{1}$ and $K\leq G_{2}$, a group $L$, and two epimorphisms
$\varphi:H\rightarrow L$ and $\psi:K\rightarrow L$ such that
\begin{equation*}
\mathscr U =\{(h,k)\in H\times K:\varphi(h)=\psi(k)\}.
\end{equation*}
The widely used notation for $\mathscr U$ is 
\begin{equation}\label{eq:amalgam-projections}
\mathscr U:=H\prescript{\varphi}{}\times_{L}^{\psi}K,
\end{equation}
in which case $H\prescript{\varphi}{}\times_{L}^{\psi}K$ is called an
\textit{amalgamated} subgroup of $G_{1}\times G_{2}$.

In this paper, we are interested in describing conjugacy classes of $\mathscr U$. Therefore, to make notation \eqref{eq:amalgam-projections} simpler and
self-contained, it is enough to indicate $L$,  
$Z=\text{Ker\thinspace}(\varphi)$ and 
$R=\text{Ker\thinspace}(\psi)$. Hence, instead of  
\eqref{eq:amalgam-projections}, we use the following notation:
\begin{equation}
\mathscr U=:H{\prescript{Z}{}\times_{L}^{R}}K~ \label{eq:amalg}.
\end{equation}

\vs\noi
{\bf (b) Equivariant Notation.} Below $\mathcal G$ stands for a compact Lie group.
For a subgroup $H$ of $\mathcal G$, 
denote by $N(H)$ the
normalizer of $H$ in $\mathcal G$ and by $W(H)=N(H)/H$ the Weyl group of $H$.  The symbol $(H)$ stands for the conjugacy class of $H$ in $\mathcal G$. 
Put $\Phi(\mathcal G):=\{(H): H\le \mathcal G\}$.
The set $\Phi (\mathcal G)$ has a natural partial order defined by 
$(H)\leq (K)$ iff $\exists g\in \mathcal G\;\;gHg^{-1}\leq K$. 
Put $\Phi_0 (\mathcal G):= \{ (H) \in \Phi(\mathcal G) \; : \; \text{$W(H)$  is finite}\}$.

For a $\mathcal G$-space $X$ and $x\in X$, denote by
$\mathcal G_{x} :=\{g\in \mathcal G:gx=x\}$  the {\it isotropy group}  of $x$
and call $(\mathcal G_{x})$   the {\it orbit type} of $x\in X$. Put $\Phi(\mathcal G,X) := \{(H) \in \Phi_0(\mathcal G) \; : \; 
(H) = (\mathcal G_x) \; \text{for some $x \in X$}\}$ and  $\Phi_0(\mathcal G,X):= \Phi(\mathcal G,X) \cap \Phi_0(\mathcal G)$. For a subgroup $H\leq \mathcal G$, the subspace $
X^{H} :=\{x\in X:\mathcal G_{x}\geq H\}$ is called the {\it $H$-fixed-point subspace} of $X$. If $Y$ is another $\mathcal G$-space, then a continuous map $f : X \to Y$ is called {\it equivariant} if $f(gx) = gf(x)$ for each $x \in X$ and $g \in \mathcal G$. 
Let $V$ be a finite-dimensional  $\mathcal G$-representation (without loss of generality, orthogonal).
Then, $V$  decomposes into a direct sum 
\begin{equation}
V=V_{0}\oplus V_{1}\oplus \dots \oplus V_{r},  \label{eq:Giso}
\end{equation}
where each component $V_{i}$ is {\it modeled} on the
irreducible $\mathcal G$-representation $\mathcal{V}_{i}$, $i=0,1,2,\dots ,r$, that is, $V_{i}$  contains all the irreducible subrepresentations of $V$
equivalent to $\mathcal{V}_{i}$. Decomposition  \eqref{eq:Giso}  is called  $\mathcal G$\textit{-isotypic  decomposition of} $V$.
\vs \noi
{\bf (b) Axioms of Equivariant Brouwer Degree.} Denote by  $\mathcal{M}^{\mathcal G}$ the set of all admissible $\mathcal G$-pairs and let $A_0(\mathcal G)$ stand for the Burnside ring of $\mathcal G$ (see Introduction, items (a) and (b) respectively). The following result (cf.  \cite{AED}) can be considered as an axiomatic definition of the {\it $\mathcal G$-equivariant Brouwer degree}.

\begin{theorem}
\label{thm:GpropDeg} There exists a unique map $\mathcal G\mbox{\rm -}\deg:\mathcal{M}
^{\mathcal G}\to A_0(\mathcal G)$, which assigns to every admissible $\mathcal G$-pair $(f,\Omega)$ an
element $\gdeg(f,\Omega)\in A_0(\mathcal G)$
\begin{equation}
\label{eq:G-deg0}\mathcal G\mbox{\rm -}\deg(f,\Omega)=\sum_{(H)}%
{n_{H}(H)}= n_{H_{1}}(H_{1})+\dots+n_{H_{m}}(H_{m}),
\end{equation}
satisfying the following properties:

\begin{itemize}
\item[] \textbf{(Existence)} If $\mathcal G\mbox{\rm -}\deg(f,\Omega)\ne
0$, i.e., $n_{H_{i}}\neq0$ for some $i$ in \eqref{eq:G-deg0}, then there
exists $x\in\Omega$ such that $f(x)=0$ and $(\mathcal G_{x})\geq(H_{i})$.

\item[] \textbf{(Additivity)} Let $\Omega_{1}$ and $\Omega_{2}$
be two disjoint open $\mathcal G$-invariant subsets of $\Omega$ such that
$f^{-1}(0)\cap\Omega\subset\Omega_{1}\cup\Omega_{2}$. Then,
\begin{align*}
\mathcal G\mbox{\rm -}\deg(f,\Omega)=\mathcal G\mbox{\rm -}\deg(f,\Omega_{1})+\mathcal G\mbox{\rm -}\deg
(f,\Omega_{2}).
\end{align*}

\item[] \textbf{(Homotopy)} If $h:[0,1]\times V\to V$ is an
$\Omega$-admissible $\mathcal G$-homotopy, then
\begin{align*}
\mathcal G\mbox{\rm -}\deg(h_{t},\Omega)=\mathrm{constant}.
\end{align*}

\item[] \textbf{(Normalization)} Let $\Omega$ be a $G$-invariant
open bounded neighborhood of $0$ in $V$. Then,
\begin{align*}
\mathcal G\mbox{\rm -}\deg(\id,\Omega)=(\mathcal G).
\end{align*}

\item[] \textbf{(Multiplicativity)} For any $(f_{1},\Omega
_{1}),(f_{2},\Omega_{2})\in\mathcal{M} ^{\mathcal G}$,
\begin{align*}
\mathcal G\mbox{\rm -}\deg(f_{1}\times f_{2},\Omega_{1}\times\Omega_{2})=
\mathcal G\mbox{\rm -}\deg(f_{1},\Omega_{1})\cdot \mathcal G\mbox{\rm -}\deg(f_{2},\Omega_{2}),
\end{align*}
where the multiplication `$\cdot$' is taken in the Burnside ring $A_0(\mathcal G )$.

\item[] \textbf{(Recurrence Formula)} For an admissible $\mathcal G$-pair
$(f,\Omega)$, the $\mathcal G$-degree \eqref{eq:G-deg0} can be computed using the
following Recurrence Formula:
\begin{equation}
\label{eq:RF-0}n_{H}=\frac{\deg(f^{H},\Omega^{H})- \sum_{(K)>(H)}{n_{K}\,
n(H,K)\, \left|  W(K)\right|  }}{\left|  W(H)\right|  },
\end{equation}
where $\left|  X\right|  $ stands for the number of elements in the set $X$
and $\deg(f^{H},\Omega^{H})$ is the Brouwer degree of the map $f^{H}%
:=f|_{V^{H}}$ on the set $\Omega^{H}\subset V^{H}$.
\end{itemize}
\end{theorem}

The $\gdeg(f,\Omega)$ is 
 called the {\it $\mathcal G$%
-equivariant  Brouwer degree of $f$ in $\Omega$}.


\vs\noi
{\bf (c) Computation of Brouwer Equivariant Degree.} 
Consider a $\mathcal G$-equivariant linear isomorphism $T:V\to V$ and assume that $V$
has a $\mathcal G$-isotypic  decomposition \eqref{eq:Giso}. Then, by the
Multiplicativity property,
\begin{equation}
\label{eq:prod-prop}\mathcal G\mbox{\rm -}\deg(T,B(V))=\prod_{i=0}^{r}\mathcal G\mbox{\rm -}\deg
(T_{i},B(V_{i}))= \prod_{i=0}^{r}\prod_{\mu\in\sigma_{-}(T)} \left(
\deg_{\mathcal{V} _{i}}\right)  ^{m_{i}(\mu)}%
\end{equation}
where $T_{i}=T|_{V_{i}}$, $\sigma_{-}(T)$ denotes the real negative
spectrum of $T$ (i.e., $\sigma_{-}(T)=\left\{  \mu\in\sigma(T):\mu<0\right\})$ and $m_i(\mu) = \dim\big(E(\mu) \cap V_i\big)$
(here $E(\mu)$ stands for the generalized eigenspace of $T$ corresponding to $\mu$). 
Notice that the basic degrees can be effectively computed from \eqref{eq:RF-0}: 
\begin{align*}
\deg_{\mathcal{V} _{i}}=\sum_{(H)}n_{H}(H),
\end{align*}
where 
\begin{equation}
\label{eq:bdeg-nL}n_{H}=\frac{(-1)^{\dim\mathcal{V} _{i}^{H}}- \sum
_{H<K}{n_{K}\, n(H,K)\, \left|  W(K)\right|  }}{\left|  W(H)\right|  }.
\end{equation}

\section{GAP package {\tt EquiDeg}}\label{sec:GAP}
In this paper we used  the GAP package {\tt EquiDeg}, developed by Hao-Pin Wu.   In order to generate the amalgamated notation for the subgroups of $G$ we applied in the  GAP code the following abbreviations for the names of subgroups of $S_4\times \bz_2$: \;\;  {\tt  Z1$:=\bz_1$, Z2$:=\bz_2$, D1z$:=D_1^z$, D1$:=D_1$,  Z2m$:=\bz_2^-$,
  Z1p$:=\bz_1\times \bz_2$, Z3$:= \bz_3$, Z2p$:=\bz_2\times \bz_2$, V4m$:=V_4^-$, D2$:=D_2$,
   Z4$:=\bz_4$, V4$:=V_4$, D2z$:=D_2^z$, Z4d$:=\bz_4^d$, D2d$:=D_2^d$, 
    D1p$:=D_1\times \bz_2$, Z3p$:=\bz_3\times \bz_2$, D3$:=D_3$, D3z$:=D_3^z$, V4p$:=V_4\times \bz_2$,
 D4d$:=D_4^d$, Z4p$:=\bz_4\times \bz_2$, D4$:=D_4$, D2p$:=D_2\times \bz_2$, D4z$:=D_4^z$, 
  D4hd$:=D_4^{\wh d}$, D3p$:=D_3\times \bz_2$, A4$:=A_4$, D4p$:=D_4\times \bz_2$, S4$:=S_4$,
  A4p$:=A_4\times \bz_2$, S4m$:=S_4^-$, S4p$:=S_4\times \bz_2$,
} where we use the same notation as in \cite{AED}, page 157.

\vs 
\noi{\small {\bf GAP code} used in this paper:\\
\begin{lstlisting}[language=GAP, frame=single]
 LoadPackage( "EquiDeg" );
# generate O(2), S4 and   Z2
  o2 := OrthogonalGroupOverReal( 2 );
  s4 := SymmetricGroup( 4 );
  z2 := pCyclicGroup( 2 );
# generate S4 x Z2
  g1 := DirectProduct( s4, z2 );
# set names for g1
  ccsg1 := ConjugacyClassesSubgroups (g1);
  ccsg1_names := [ "Z1", "Z2", "D1z", "D1", "Z2m",
  "Z1p", "Z3", "Z2p", "V4m", "D2",
  "Z4", "V4", "D2z", "Z4d", "D2d", 
  "D1p", "Z3p", "D3", "D3z", "V4p",
  "D4d", "Z4p", "D4", "D2p", "D4z", 
  "D4hd", "D3p", "A4", "D4p", "S4",
  "A4p", "S4m", "S4p"];
  ListA( ccsg1, ccsg1_names, SetAbbrv );
# generate the group G=O(2)xS4xZ2
 G := DirectProduct( o2, g1 );
  ccs := ConjugacyClassesSubgroups( G );
# Character Table for S4xZ2
  tbl := CharacterTable (  g1  );
  Display( tbl );
  Display( ConjugacyClasses(  g1  ));
# unit element in AG
  u:=-BasicDegree(Irr( G )[0,1]);
# Compute basic degrees deg[l,k] for W_l\times U_k^-
  deg10:=BasicDegree( Irr( G )[1,3]);
  deg20:=BasicDegree( Irr( G )[2,3]);
  deg30:=BasicDegree( Irr( G )[3,3]);
  deg03:=BasicDegree( Irr( G )[0,7]);
  deg13:=BasicDegree( Irr( G )[1,7]);
  deg04:=BasicDegree( Irr( G )[0,8]);  
# determining maximal orbit types in M_1
  maxorbit10:=MaximalOrbitTypes(Irr( G )[1,3]);
  maxorbit11:=MaximalOrbitTypes( Irr( G )[1,2]);
  maxorbit13:=MaximalOrbitTypes( Irr( G )[1,7]);
  maxorbit14:=MaximalOrbitTypes( Irr( G )[1,8]);
  M1 := MaximalElements(Union(maxorbit10,
maxorbit11,maxorbit13,maxorbit14));
# identifying the maximal orbit types in M_1
  IdCCS(M1[1]);View(M1[1]); 
  IdCCS(M1[2]);View(M1[2]);
  IdCCS(M1[3]);View(M1[3]);
# Compute deg of F
 deg:=u - deg10*deg20*deg30*deg03*deg13*deg04;
# get amalgamation symbols for the related orbit types, e.g. for H[1,131] by
  ccs[1,131];
\end{lstlisting}}

%

\end{document}